\documentclass[12pt]{amsart}

\newtheorem{theorem}{Theorem}[section]

\newtheorem{corollary}[theorem]{Corollary}

\theoremstyle{definition}

\newtheorem{remark}[theorem]{Remark}

\theoremstyle{parrafo}

\begin{document}

\title[]{A monotonicity property of variances}

\author{J. M. Aldaz}
\address{Departamento de Matem\'aticas,
Universidad  Aut\'onoma de Madrid, Cantoblanco 28049, Madrid, Spain.}
\email{jesus.munarriz@uam.es}

\thanks{The author was partially supported by Grant MTM2012-37436-C02-02  from the
MINECO of Spain}







\begin{abstract} We prove that variances of non-negative random variables have the following monotonicity property:
For all $0 < r < s \le 1$,  and all $0 \le X \in L^2$, we have 
$\operatorname{Var}(X^r)^{1/r} \le \operatorname{Var}(X^s)^{1/s}$. 
We also discuss the   real valued case.
\end{abstract}


\maketitle


\markboth{J. M. Aldaz}{AM-GM}

\section{Introduction} Here, statements such as $X \ge 0$ or $X = Y$, are always meant
in the almost sure sense.
It is immediate from either H\"older's or Jensen's inequality that
for every random variable $X \ge 0$ and all $0 < r < s < \infty$, we have
$(E X^r)^{1/r} \le (E X^s)^{1/s}$. 
In this note we obtain an analogous result for non-negative random variables
$X\in L^2$ and variances. As in the case of norms, this inequality helps to clarify the strength of 
hypotheses that might be made on $\operatorname{Var}(X^r)$. An application to a recent refinement of  the  
AM-GM inequality
$
\prod_{i=1}^n x_i^{\alpha_i}  
\le 
\sum_{i=1}^n \alpha_i x_i
$
is  presented. Lastly, this monotonicity property can be used 
when dealing with real valued random variables, by decomposing them into their positive and
negative parts, since the variance of $X$ is always comparable
to the sum of the variances of $X_+$ and $X_-$.

\section{Monotonicity of $\operatorname{Var}(X^s)^{1/s}$, and the AM-GM inequality.}
Let $0 \le X\in L^2$, so $\operatorname{Var}(X)$ is well defined. Since for all $0 < s \le 1$ we have
$\|X\|_{2s} \le \|X\|_2$, all variances $\operatorname{Var}(X^s)$ are also well defined, and thus it
is natural to ask how these quantities behave as $s$ changes. In order to be able to compare them, we
need to have the same homogeneity on both sides of the inequality, so we consider $\operatorname{Var}(X^s)^{1/s}$,
which always is homogeneous of order 2: For all $t \ge 0$,  
$\operatorname{Var}((t X)^s)^{1/s} = t^2 \operatorname{Var}(X^s)^{1/s}$.

\begin{theorem}\label{MonVar} Let $0 \le X\in L^2$ and let $0 < r < s \le 1$. Then 

\begin{equation}\label{incVar}
\operatorname{Var}(X^r)^{1/r} \le \operatorname{Var}(X^s)^{1/s}.  
\end{equation}
\end{theorem}

\begin{proof} Observe first that it is enough to prove the case $\operatorname{Var}(X^s)^{1/s} \le \operatorname{Var}(X)$
whenever $0 <s<1$.
The fact that $\operatorname{Var}(X^s)^{1/s}$ is increasing in $s$ then follows immediately by making the change of variables
$Y = X^s$:
$
\operatorname{Var}(X^r)^{s/r} = \operatorname{Var}(Y^{r/s})^{s/r}  \le \operatorname{Var}(Y)
 =\operatorname{Var}(X^s).
$ 

Next, we assume that $\|X\|_2 = 1$. This can be done by homogeneity, since writing $Y = X/\|X\|_2$,
we see that $\operatorname{Var}(X^s)^{1/s} \le \operatorname{Var}(X)$ is equivalent to 
$\operatorname{Var}(Y^s)^{1/s} \le \operatorname{Var}(Y)$. Under the condition $\|X\|_2 = 1$, we always have, 
for every $0 < s \le 1$
and every $t > 0$, $\|X\|_{2s}^t \le 1$, and hence, $\operatorname{Var}(X^s)^{t} \le 1$.

We shall use the following well known (and direct) interpolation consequence of H\"older's inequality
(cf., for instance, \cite[Proposition 6.10, p. 177]{Fo}) which is valid for both finite and infinite measure
spaces:
If $0 < r < s < p$, and $f\in L^r\cap L^p$, then $f$ belongs to all intermediate spaces $L^s$, and
furthermore, $\|f\|_s \le \|f\|_r^{1-t} \|f\|_p^{t}$, where $t\in (0,1)$ is defined by the equation
$1/s = (1-t)/r + t/p$. Using the indices $0 < s < 2s <2$, together with $\|X\|_2 = 1$, yields $t = 1/(2-s)$ and
\begin{equation}\label{interp1}
E (X^{2s}) \le (E X^s)^{(2 - 2s)/(2-s)},   
\end{equation}
while the indices $0 < s < 1 <2$ give  $t = (2 - 2s)/(2-s)$ and
\begin{equation}\label{interp2}
E (X) \le (E X^s)^{1/(2-s)}.   
\end{equation}

Now, by the preceding assumptions on the size of norms and variances (in particular, by
$\|X^s\|_{2}^2  = \|X\|_{2s}^{2s} \le 1 $) together with $1/s >1$, we have
$$ \operatorname{Var}(X^s)^{1/s} \le \operatorname{Var}(X^s) 
= \|X^s\|_{2}^2 \operatorname{Var}\left(\frac{X^s}{\|X^s\|_{2}}\right)
\le \operatorname{Var}\left(\frac{X^s}{\|X^s\|_{2}}\right) = 1 - \frac{(E X^s)^2}{E(X^{2s})}.
$$
Thus, it suffices to show that 
$$ 1 - \frac{(E X^s)^2}{E(X^{2s})} \le \operatorname{Var}(X) = 1 - (E X)^2,
$$
or equivalently, that
$$(E X)^2 E(X^{2s}) \le (E X^s)^2.
$$
But this follows from (\ref{interp2}) and (\ref{interp1}), since 
$$(E X)^2 E(X^{2s}) \le (E X^s)^{2/(2-s)} (E X^s)^{(2 - 2s)/(2-s)} = (E X^s)^2.
$$
\end{proof}

\begin{remark} The  interpolation result noted above is useful in a
probability context since, instead of the usual bound $\|X\|_s \le \|X\|_p$ whenever
$0 < s < p$, it yields  the stronger inequality $\|X\|_s \le \|X\|_r^{1-t} \|X\|_p^{t}$ for each
$0 < r  < s$, with $t$ defined by $1/s = (1-t)/r + t/p$.
\end{remark}

Of course, under different integrability conditions ($X\in L^p$  instead of
$X\in L^2$) the analogous inequalities hold, by using the change of variables
$Y = X^{p/2} \in L^2$.

\begin{corollary}\label{MonotVar} Let $p > 0$, let $0 \le X\in L^p$, and let $0 < r < s \le p/2$. Then 

\begin{equation}\label{increVar}
\operatorname{Var}(X^r)^{1/r} \le \operatorname{Var}(X^s)^{1/s}.  
\end{equation}
\end{corollary}

Next we apply the preceding result to a recent refinement of the inequality between arithmetic and geometric 
means (the AM-GM inequality) proven in \cite{A1} (the reader interested in some probabilistic aspects of
the AM-GM inequality, may want to consult \cite{A3} and the references contained therein; for non-variance bounds,
see \cite{A4} and its references). 
Let us recall the notation used in \cite{A1}:
$X$  denotes the vector with non-negative
entries $(x_1,\dots,x_n)$, and $X^{1/2} = (x_1^{1/2},\dots,x_n^{1/2})$.
Given a sequence of weights $\alpha = (\alpha_1,\dots,\alpha_n)$ with $\alpha_i > 0$ and $\sum_{i=1}^n \alpha_i = 1$, 
and a vector $Y = (y_1,\dots,y_n)$, its arithmetic mean is denoted by $E_\alpha(Y) := 
\sum_{i=1}^n \alpha_i y_i$,  its geometric mean, by $\Pi_\alpha(Y) := \prod_{i=1}^n y_i^{\alpha_i} $,
and its   variance, 
by 
$$
\operatorname{Var}_\alpha (Y) = \sum_{i=1}^n \alpha_i \left(y_i - \sum_{k=1}^n \alpha_k y_k \right)^2
= \sum_{i=1}^n \alpha_i y_i^2 - \left(\sum_{k=1}^n \alpha_k y_k \right)^2.
$$ 
 Finally, $Y_{\max}$ and
$Y_{\min}$ respectively stand for the maximum and the minimum values of $Y$.

Conceptually, variance bounds for $E_\alpha X - \Pi_\alpha X$
represent the natural extension of the equality case in the AM-GM
inequality (zero variance is equivalent to equality). From a more applied viewpoint,
the variance is used in the
Economics literature to estimate 
the difference between these means  (cf., for instance, \cite[Chapter 1, Appendix 2]{Si}; both the 
arithmetic and geometric means are used
when reporting on the performance of
a portfolio).

The bounds for the difference in the AM-GM appearing in  \cite{A1} involve $\operatorname{Var}(X^{1/2})$, rather than
$\sigma(X) = \operatorname{Var_\alpha}(X)^{1/2}$. Using Theorem \ref{MonVar} or Corollary \ref{MonotVar}, the following upper
bound follows: $E_\alpha X - \Pi_\alpha X
\le
\frac{1}{\alpha_{\min}} \sigma(X)$. More generally, 
by putting together \cite[Theorem 4.2]{A1} with Corollary \ref{MonotVar}, we obtain the next result.

\begin{theorem}\label{AMGMVar1}  For $n\ge 2$ and $i=1,\dots, n$, let $X = (x_1,\dots,x_n)$ be such that $x_i\ge 0$, and
let $\alpha = (\alpha_1,\dots,\alpha_n)$ satisfy
$\alpha_i > 0$ and $\sum_{i=1}^n \alpha_i = 1$.  Then for all $r\in (0,1]$ and all $s\in [1,\infty)$ we have
\begin{equation}\label{AMGMVar1eq}
\frac{1}{1-\alpha_{\min}}\operatorname{Var}_\alpha(X^{r/2})^{1/r} \le E_\alpha X - \Pi_\alpha X
\le
\frac{1}{\alpha_{\min}} \operatorname{Var_\alpha}(X^{s/2})^{1/s}.
\end{equation}
\end{theorem}

These bounds are optimal (cf. \cite[Examples 2.1 and 2.3]{A1}).  Theorem 4.2 from 
 \cite{A1}, and its proof, were suggested by 
 \cite[Theorem]{CaFi}, which states that if $0 < X_{\min}$,
then
\begin{equation}\label{cafieq}
\frac{1}{2 X_{\max} }  \operatorname{Var}_\alpha(X)
\le
E_\alpha X - \Pi {}_\alpha X
\le
\frac{1}{2 X_{\min}}  \operatorname{Var}_\alpha(X).
\end{equation}
A drawback of  (\ref{cafieq}) is that the bounds  depend explicitly on $X_{\max}$ and $X_{\min}$, something
that makes it unsuitable for some standard applications, such as, for instance, 
refining H\"older's inequality  (see  \cite{A1} for more details). Of course, since the  variance
 is homogeneous of degree 2, dividing by $ X_{\max}$ and $ X_{\min}$  in (\ref{cafieq}),
gives the left and right hand sides
the same homogeneity as the middle term.  We also point out that the  inequality 
$\operatorname{Var}_\alpha(X^{1/2}) \le E_\alpha X - \Pi_\alpha X$, appeared
in \cite[Theorem 1]{A2}; this inequality is trivial, useful, and as $n\to\infty$,
asymptotically
optimal, since $(1-\alpha_{\min})^{-1}\to 1$.

\section{Real valued random variables.}

The monotonicity result applies to $X\ge 0$ only: If  $X < 0$ with positive probability,
then $X^s$
may fail to be defined as a real valued function, for certain values of $s  > 0$. While trivially
$\operatorname{Var}(X)\ge \operatorname{Var}(|X|)$, 
in general these two quantities are not comparable, so it is not possible to simply replace $X$ with $|X|$. 
However,  monotonicity 
can  be used on $\operatorname{Var}(X_+)$ and  $ \operatorname{Var}(X_-)$, where $X_+ := \max\{X, 0\}$ and 
$X_-:=  - \min\{X, 0\}$   denote the positive and negative parts of $X$, respectively. Thus, indirectly it also
applies to 
$\operatorname{Var}(X)$, since the latter is indeed comparable
to $\operatorname{Var}(X_+) + \operatorname{Var}(X_-)$. 
We have not found this result in the literature, so we include it here for completeness.
Essentially, the next theorem says that 
\begin{equation*}
\operatorname{Var}(X_+) + \operatorname{Var}(X_-)
 \le \operatorname{Var}(X)
\le   2\left(\operatorname{Var}(X_+) + \operatorname{Var}(X_-)\right),
\end{equation*}
and the extremal cases occur,  for the left hand side inequality, when either $X\ge 0$ or $X\le 0$, and
for the right hand side inequality,
when  $X = c(\mathbf{1}_{D} - \mathbf{1}_{D^c})$, where $c\in\mathbb{R}$ and $D$ is a measurable set.

\begin{theorem}\label{posnegteor} Let  $X\in L^2$ be real valued,  and denote by  $\mathcal{B}$ the
sub-$\sigma$-algebra 
$$\mathcal{B}:= \{\emptyset, \Omega, \{X > 0\},  \{X= 0\}, \{X < 0\}\}.
$$
Then
\begin{equation}\label{posneg}
\operatorname{Var}(X_+) + \operatorname{Var}(X_-) \le \operatorname{Var}(X)
\end{equation}
\begin{equation}\label{posneg2}
\le  \operatorname{Var}(X_+) + \operatorname{Var}(X_-) + \operatorname{Var}(E(X_+|\mathcal{B})) 
+ \operatorname{Var}(E(X_-|\mathcal{B}))
\le 2\left(\operatorname{Var}(X_+) + \operatorname{Var}(X_-)\right).
\end{equation}
Furthermore, equality holds in the first inequality if and only if either $X\ge 0 $ or $X\le 0 $;
in the second, if and only if either  $X > 0 $,  or $X < 0 $, or $0 < P(\{X > 0\})$,  $0 <  P(\{X < 0\})$,
$0 = P(\{X > 0\})$, and 
$E(X_+|\{X > 0\}) = E(X_-|\{X < 0\})$;
and in the third, if and only if  $X = E(X|\mathcal{B})$.
\end{theorem}

\begin{proof} The first inequality follows directly from the definitions, the second, from the convexity of
$\phi ( x) = x^2$, and  the third, from the law of total variance. More
precisely,
$$
\operatorname{Var}(X_+) + \operatorname{Var}(X_-) 
\le \operatorname{Var}(X_+) + \operatorname{Var}(X_-) + 2 EX_+EX_-
$$
$$
= E(X_+^2) - (EX_+)^2 + E(X_-^2) - (EX_-)^2 + 2 EX_+EX_- =  E(X^2) - (EX_+ - EX_-)^2
=\operatorname{Var}(X),
$$ 
and we have equality if and only if $EX_+EX_ - = 0$, which happens if and only if either
$X\ge 0$ or $X\le 0$.

Since, as we just saw, $\operatorname{Var}(X) = \operatorname{Var}(X_+) + \operatorname{Var}(X_-) + 2 EX_+EX_-$,
to prove the middle inequality in (\ref{posneg})-(\ref{posneg2}),
it is enough to show that 
\begin{equation}\label{posneg3}
2 EX_+EX_- \le  \operatorname{Var}(E(X_+|\mathcal{B})) 
+ \operatorname{Var}(E(X_-|\mathcal{B})).
\end{equation}

Observe that if either
$X\ge 0$ or $X\le 0$, then 
$$
2 EX_+EX_-
 = 0,
$$ 
and if additionally  either
$X> 0$ or $X < 0$, then
$$
 0 = \operatorname{Var}(E(X_+|\mathcal{B})) 
+ \operatorname{Var}(E(X_-|\mathcal{B})).
$$ 
Next,  assume that both
$A:= P\{X > 0\} >0$ and   $B:= P\{X < 0\} > 0$, and write $C:= P\{X =0\}$, so $0 < A + B = 1 - C \le 1$.
Then $E(X|\mathcal{B})$ takes exactly two values different from 0,
say $E(X|\mathcal{B}) = a > 0$ on $\{X > 0\}$, 
and $E(X|\mathcal{B}) = -b < 0$ on $\{X < 0\}$.
With this notation, in order to obtain the middle inequality it suffices to show that 
\begin{equation*} 2 EX_+EX_- 
= 2 Aa Bb \le 
 \operatorname{Var}\left( E(X_+|\mathcal{B})\right)  +  \operatorname{Var}\left( E(X_-|\mathcal{B})\right)
= A a^2 - (Aa)^2 + B b^2 -(Bb)^2 ,
\end{equation*}
or equivalently, that
$$
(Aa + Bb)^2 \le  A a^2  + B b^2.
$$
But this is follows from the convexity of $\phi (x) = x^2$, since 
$$
(Aa + Bb)^2 
= (A + B)^2 \left(\frac{A}{A+B}a + \frac{B}{A+B}b\right)^2 
\le  (A + B)^2 \left(\frac{A}{A+B}a^2 + \frac{B}{A+B}b^2\right) 
$$
$$
= (A + B)\left(Aa^2 + Bb^2\right) 
\le  
A a^2  + B b^2.
$$
Furthermore, 
$
(Aa + Bb)^2 
= 
A a^2  + B b^2
$ 
if and only if both  $a = b$ (by  the strict convexity of $\phi$) and $A + B=1$.

Finally,  the law of total variance 
$\operatorname{Var}(X) = \operatorname{Var}\left( E(X|\mathcal{B})\right) 
+ E\left( \operatorname{Var}(X|\mathcal{B})\right)$, applied to both $X_+$ and $X_-$, tells
us that
$
\operatorname{Var}\left( X_+\right)  \ge \operatorname{Var}\left( E(X_+|\mathcal{B})\right)$
and
$
\operatorname{Var}\left( X_-\right)  \ge \operatorname{Var}\left( E(X_-|\mathcal{B})\right)$, 
with equality if and only if 
$E\left( \operatorname{Var}(X_+|\mathcal{B})\right) = 0 =E\left( \operatorname{Var}(X_-|\mathcal{B})\right)$,
which happens if and only if both $X_+$ and $X_-$ are constant on $\{X > 0\}$ and on $\{X < 0\}$
respectively. This yields the last inequality, together with the  equality condition $X = E(X|\mathcal{B})$.
\end{proof}

\begin{remark} Instead of 
$\mathcal{B}= \{\emptyset, \Omega, \{X > 0\},  \{X= 0\}, \{X < 0\}\},
$ either of the simpler algebras $\mathcal{B}_1= \{\emptyset, \Omega, \{X \ge 0\},   \{X < 0\}\}
$ or $\mathcal{B}_2= \{\emptyset, \Omega, \{X > 0\},  \{X \le 0\}\}
$ could have been used in the preceding theorem, and the inequalities stated there would still hold.
But the equality conditions would be less symmetric. For instance, if $X\ge 0$, then 
$\mathcal{B}_1$ is trivial up to sets of measure zero (that is, as a measure algebra), so 
$E(X_+|\mathcal{B}_1) = EX_+ = EX$, and $\operatorname{Var}\left( E(X_+|\mathcal{B}_1)\right) = 0$.
Thus, the middle inequality in (\ref{posneg})-(\ref{posneg2}), is actually an equality in this case.
However, if $X =  - \mathbf{1}_{D}\le 0$, where $0 < P(D) < 1$, then $X = X_- = E(X_-|\mathcal{B}_1)$,
and $ \operatorname{Var}(X)
< 
\operatorname{Var}(X_-) +\operatorname{Var}(E(X_-|\mathcal{B}_1)) 
= 2 \operatorname{Var}(X).
$
\end{remark}

\begin{corollary} Let $p \ge 2$, let $X\in L^p$ be real valued, and let $0 < r \le 2 \le s \le p$. Then 

\begin{equation*}
\operatorname{Var}(X^{r/2}_+ )^{2/r}  + \operatorname{Var}(X^{r/2}_- )^{2/r}
 \le \operatorname{Var}(X)
 \le 2 \left(\operatorname{Var}(X_+^{s/2})^{2/s} + \operatorname{Var}(X_-^{s/2})^{2/s}\right).  
\end{equation*}
\end{corollary}

\end{document}